\documentclass[reqno]{amsart}
\usepackage{hyperref}

\usepackage[]{graphicx}

\usepackage{amsthm}
\usepackage{amsmath}
\usepackage{latexsym}
\usepackage{amsfonts}
\usepackage{amssymb}

\newcommand{\CC}{{\mathbb C}}
\newcommand{\RR}{{\mathbb R}}

\newcommand{\DD}{{\mathbb D}}

\newcommand{\LL}{{\mathcal L}}
\newcommand{\GG}{{\mathcal G}}

\newcommand{\e}{\varepsilon}
\newcommand{\be}{\beta}
\newcommand{\al}{\alpha}
 \newcommand{\De}{\Delta}
\newcommand{\la}{\lambda}
\newcommand{\g}{\gamma}
\newcommand{\de}{\delta}
\newcommand{\ph}{\varphi}
\newcommand{\Ga}{\Gamma}
\newcommand{\Om}{\Omega}
\newcommand{\si}{\sigma}

\newcommand{\FF}{\mathcal{F}}

\newcommand{\EEE}{\mathbb{E}}

\newcommand{\ii}{\mbox{i}}

\begin{document}
\title[Subordination principle for space-time fractional equations]
{ Subordination principle for space-time fractional evolution equations and some applications }

\author{Emilia Bazhlekova}
\address{Institute of Mathematics and
Informatics, Bulgarian Academy of Sciences, Acad. G. Bonchev str., Bl. 8, Sofia 1113, Bulgaria}
\email{e.bazhlekova@math.bas.bg}


\begin{abstract}
 The abstract Cauchy problem for the fractional evolution equation with the Caputo derivative of order $\be\in(0,1)$ and operator $-A^\al$, $\al\in(0,1)$, is considered, where $-A$ generates a strongly continuous one-parameter semigroup on a Banach space. Subordination formulae for the solution operator are derived, which are integral representations containing a subordination kernel (a scalar probability density function) and a $C_0$-semigroup of operators. Some properties of the subordination kernel are established and representations in terms of Mainardi function and L\'evy extremal stable densities are derived. 
			Applications of the subordination formulae are given with a special focus on the multi-dimensional space-time fractional diffusion equation for some special values of the parameters.
\end{abstract}

\maketitle

\numberwithin{equation}{section}
\newtheorem{theorem}{Theorem}[section]
\newtheorem{lemma}[theorem]{Lemma}
\newtheorem{example}[theorem]{Example}
\newtheorem{corollary}[theorem]{Corollary}
\newtheorem{prop}[theorem]{Proposition}
\newtheorem{definition}[theorem]{Definition}
\newtheorem{remark}[theorem]{Remark}
\newtheorem{proposition}[theorem]{Proposition}

\section{ Introduction}

Partial differential equations with both time- and space-fractional differential operators have found numerous applications in the modelling of anomalous diffusion phenomena  
and are nowadays a subject of extensive research, see e.g. 
\cite{Meerschaert, Bonito, Pasciak, LuchkoCAMWA, Chen, GMainardi, GMsub, Hanyga, LG, LuchkoMMNP, LuchkoAxioms, Luchko, Luchkosub,  spacetimepdf, MainardiPagniniGorenflo, Meerschaert2002, Pagnini}, 
to mention only an excerpt from the long list of relevant publications.

Consider the space-time fractional diffusion equation 
\begin{equation}\label{st}
	\DD_t^\be u(x,t)=-(-\De)^\al u (x,t),\ \ \ \ t>0, x\in\RR^n;\ \ \ \ u(x,0)=v(x);
	\end{equation}
where $0<\al,\be\le 1$, $\DD_t^\be$ is the Caputo time-fractional derivative and $-(-\De)^\al$ denotes a realization of the fractional Laplace operator acting in space.
  
	There exist different definitions of the Laplacian and
	 the fractional Laplacian in the literature, see e.g.
\cite{Bonito, Y2012, critical, FCAA10, FL, Martinez, Y2010}. Once the Laplace operator $\De$ is defined, one of the commonly used definitions for the fractional Laplacian is provided by the fractional power of the nonnegative 
operator $-\De$ according to the Balakrishnan definition \cite{Balakrishnan, Yosida}
\begin{equation}\label{Deal}
(-\De)^\al v=\frac{\sin \al\pi}{\pi}\int_0^\infty \la^{\al-1}(\la -\De)^{-1}(-\De) v\, d\la,\ \ 0<\al<1,
\end{equation}
for $v\in D(\De)$ - the domain of the considered Laplace operator.
In the case of the full-space fractional Laplace operator in $\RR^n$, the defined in (\ref{Deal}) fractional power $(-\De)^\al$ coincides with the following pseudo-differential operator 
$$
(\FF(-\De)^\al f)(\kappa)=|\kappa|^{2\al}(\FF f)(\kappa),
$$
where $\FF$ denotes the Fourier transform.
 In particular, in the one-dimensional case $-(-\De)^\al$ coincides with the Riesz space-fractional derivative of order $2\al$. We refer to  the survey paper \cite{FCAA10}, where the equivalence of ten different definitions of full-space fractional Laplacian is proven. 
Concerning the fractional Laplacian on a bounded domain $\Om\subset\RR^n$, frequently used is the spectral definition (suitable as well for more general elliptic boundary value operators), according to which the fractional power $(-\De)^\al$ is defined in terms of Fourier series as follows
\begin{equation}\label{fl}
(-\De)^\al v=\sum_{j=1}^\infty \la_j^\al (v,\phi_j)\phi_j.
\end{equation}
Here $(.,.)$ denotes the $L^2(\Om)$ inner product and $\phi_j$ is an $L^2(\Om)$-orthonormal basis of eigenfunctions of $-\De$ with eigenvalues $\la_j$.
The spectral definition is equivalent to the Balakrishnan definition in $L^2(\Om)$, see e.g. \cite{Bonito, Pasciak, Duan}.  

The space-time fractional diffusion equation (\ref{st}) is extensively studied. The Cauchy problem for the spatially one-dimensional equation is analyzed in 
\cite{ Meerschaert, LG, LuchkoAxioms, Luchko, Luchkosub, spacetimepdf, MainardiPagniniGorenflo, Pagnini}. 
The multi-dimensional space-time fractional diffusion equation on $\RR^n$ is considered in \cite{LuchkoCAMWA, Hanyga, LuchkoMMNP, Luchko, Luchkosub}, see also \cite{GMsub, Meerschaert2002} for studies in the context of stochastic solutions and inverse stable subordinators. Multi-dimensional problems on a finite domain are discussed in \cite{Meerschaert, Pasciak, Chen}.

By means of a subordination formula it is possible to construct new solutions from known ones, e.g. solutions of fractional order equations from the solutions of the classical diffusion or wave equations. 
It is worth noting that the principle of subordination is closely related to the concept of subordination in stochastic processes  \cite{Feller}. 
Subordination formulae for the one-dimensional space-time fractional diffusion equation are established in \cite{GMsub, MainardiPagniniGorenflo, Pagnini}.
In \cite{Luchkosub} subordination principles for the multi-dimensional space-time fractional diffusion-wave equation are deduced, where the subordination kernel is expressed in terms of four-parameters Wright functions, introduced in \cite{LG}. In the above works the technique of Mellin transform is applied  for derivation of the subordination formulae.
In the setting of abstract Cauchy problems a subordination formula is established in \cite{Yosida}, which relates the $C_0$ - semigroups generated by the operators  $-A$ and $-A^\al$, $\al\in(0,1)$, where the fractional power is in the sense of the Balakrishnan definition. Concerning time-fractional evolution equations, the subordination principle is studied in \cite{Subordination0, Baj} and  extended to equations with more general time-fractional operators in \cite{K, ITSF, Math, JCAM, FCAA18}.  In \cite{MiaoLi} a generalized subordination principle is discussed for fractional evolution equations with operator $-A^\al$, where $-A$ is a generator of a fractional resolvent family. Subordination principle in the setting of abstract Volterra equations is studied in \cite{Pruss}, Chapter~4.

This work is concerned with the subordination principle for 
the abstract Cauchy problem for the space-time fractional evolution equation 
	\begin{equation}\label{ab0}
	\DD_t^\be u(t)=-A^\al u (t),\ t>0;\ \ u(0)=v\in X;\ \ \ \ 0<\al,\be\le 1,
	\end{equation}
where 
$-A$ is a generator of a bounded $C_0$-semigroup in a Banach space $X$ and 
$A^\al$ denotes the $\al$-th fractional power of the operator $A$ according to the Balakrishnan definition, see (\ref{Aal}).
Applying successively known subordination results in space and in time, we derive a subordination formula for the solution operator $S_{\al,\be}(t)$ of problem (\ref{ab0}) in the form
\begin{equation}\label{sub0}
S_{\al,\be}(t)=\int_0^\infty \psi_{\al,\be}(t,\tau) S_{1,1}(\tau)\,d \tau,\ \ t>0,
\end{equation}
where $S_{1,1}(t)$ is the $C_0$- semigroup of operators generated by the operator $-A$ 
and $\psi_{\al,\be}(t,\tau)$ is a unilateral probability density function (pdf) in $\tau$, which means that
\begin{equation}\label{pdf}
\psi_{\al,\be}(t,\tau)\ge 0, \ \ \int_0^\infty \psi_{\al,\be}(t,\tau) \,d \tau=1.
\end{equation}
It appears that the subordination kernel $\psi_{\al,\be}(t,\tau)$ 
can be expressed in terms of  
L\'evy one-sided stable extremal pdf and Mainardi function. 
Moreover, we prove that the  kernel $\psi_{\al,\be}(t,\tau)$ considered as a function of $t$ admits a bounded analytic extension to a sector in the complex plane, which implies analyticity of the solution operator  $S_{\al,\be}(t)$ at least in the same sector. By applying the subordination formula to the multi-dimensional space-time fractional diffusion equation for special values of the parameters, we recover some known expressions for the solution, but also derive some analytical formulae and asymptotic expressions in the case $\al=\be=1/2$, which, to the best knowledge of the author, are new. 
The main tool in this work is the familiar technique of Laplace transform.


The paper is organized as follows. In Section~2 the main subordination formula is derived and some properties of the subordination kernel are established. Section~3 is devoted to scaling properties of the subordination kernel and relation to L\'evy extremal stable density and Mainardi function. In Section~4 the subordination formulae are applied  to find analytic solutions to three different problems.  Definitions of fractional calculus operators, Mittag-Leffler functions and some other special functions used in this paper, together with their basic properties, are given in an Appendix.




\section{Subordination formulae}

The main tool in this work is the Laplace transform  $$\mathcal{L}\{f(t); t\to s\}=\int_0^\infty e^{-st} f(t)\, d t.$$
For convenience the following notation  for 
 the double Laplace transform of a function of two variables is also used
$$ \mathcal{L}^2\{f(t,\tau); t\to s,\tau\to\la\}=\int_0^\infty  \int_0^\infty e^{-(st+\la\tau)}f(t,\tau)\, d t\,d\tau.$$

Consider a  Banach space $X$ with norm $\|.\|$ and a closed linear operator $A$ with dense domain $D(A)\subset X$. 
Denote by $\varrho(A)$ the resolvent set of $A$. 
Assume the operator $-A$ is the infinitesimal generator of a bounded  $C_0$-semigroup (see e.g. \cite{Laplace, Yosida}).  Therefore, $A$ is a non-negative operator, i.e.  $(-\infty, 0)\subset\varrho(A)$ and
$$
\|\la(\la+A)^{-1}\|<M<\infty,\ \ \la>0.
$$
For $0<\al<1$ we define the fractional power $A^\al$ of the non-negative operator $A$ using the Balakrishnan definition \cite{Balakrishnan, Yosida}
\begin{equation}\label{Aal}
A^\al v=\frac{\sin \al\pi}{\pi}\int_0^\infty \la^{\al-1}(\la +A)^{-1}Av\, d\la,\ \ v\in D(A).
\end{equation}
Then $-A^\al$ is a closed densely defined operator, which generates a bounded analytic $C_0$-semigroup \cite{Yosida}.
Therefore the Cauchy problem for the fractional evolution equation
\begin{equation}\label{ab}
	\DD_t^\be u(t)=-A^\al u (t),\ t>0;\ \ u(0)=v\in X;\ \ \ \ \ 0<\al\le 1,\ \ 0<\be\le 1,
	\end{equation}
where $\DD_t^\be$ is the Caputo time-fractional derivative, is well posed (\cite{Baj}, Theorem~3.1).
The definitions of well-posedness, strong solution, and solution operator for problem (\ref{ab}) are standard, see e.g. \cite{Baj}.

To derive a subordination formula relating the solution operator $S_{\al,\be}(t)$ of problem (\ref{ab}) with the solution operator $S_{1,1}(t)$ of the classical abstract Cauchy problem 
\begin{equation*}
 u'(t)=-Au(t),\ \ t>0;\ \ \ \ u(0)=v\in X.
\end{equation*}
  we apply successively two known subordination results. 

First, let us set $\be=1$ in (\ref{ab}) and apply a classical theorem (see \cite{Yosida}, Chapter~IX) according to which the operator $-A^\al$ generates a bounded analytic semigroup $S_{\al,1}(t)$, related to the semigroup $S_{1,1}(t)$ via the identity
\begin{equation}\label{suba}
S_{\al,1}(t)=\int_0^\infty f_{\al}(t,\tau) S_{1,1}(\tau)\,d \tau,\ \ t>0,
\end{equation}
where the subordination kernel $f_{\al}(t,\tau)$ is defined by the inverse Laplace integral
\begin{equation}\label{fdef}
f_{\al}(t,\tau)=\frac{1}{2\pi \ii}\int_{\si-\ii \infty}^{\si+\ii \infty} e^{z\tau-t z^\al}\,d z,\ \ \si>0.
\end{equation}
The semigroup $S_{\al,1}(t)$ is the solution operator to the Cauchy problem
\begin{equation}\label{a}
 u'(t)=-A^\al u(t),\ \ t>0;\ \ \ \ u(0)=v\in X.
\end{equation}
It is worth noting that in the scalar case $A=\la>0$ relation (\ref{suba}) reads
\begin{equation}\label{subas}
e^{-\la^\al t}=\int_0^\infty f_{\al}(t,\tau) e^{-\la \tau}\,d \tau,\ \ t>0,
\end{equation}
which is in agreement with the definition (\ref{fdef}) of the kernel $f_{\al}(t,\tau)$ as the inverse Laplace transform of $e^{-\la^\al t}$.

Second, according to the subordination principle for fractional evolution equations \cite{Baj, Subordination0}, the well-posedness of problem (\ref{a}) implies well-posedness of problem (\ref{ab}) for all $\be\in(0,1)$ and the corresponding solution operator $S_{\al,\be}(t)$ is expressed by the formula
 \begin{equation}\label{subab}
S_{\al,\be}(t)=\int_0^\infty \ph_{\be}(t,\tau) S_{\al,1}(\tau)\,d \tau,\ \ t>0,
\end{equation}
where 
\begin{equation}\label{phidef}
\ph_{\be}(t,\tau)=\frac{1}{2\pi \ii}\int_{\si-\ii \infty}^{\si+\ii \infty} z^{\be-1}e^{zt-\tau z^\be}\,d z,\ \ \si>0.
\end{equation}
Since in the scalar case $A=\la>0$ the solution operator $S_{\al,\be}(t)$ of problem (\ref{ab}) is given by the Mittag-Leffler function 
$
E_\be(-\la^\al t^\be),
$
the scalar version of relation (\ref{subab}) is
\begin{equation}\label{subabs}
E_\be(-\la^\al t^\be)=\int_0^\infty \ph_{\be}(t,\tau) e^{-\la^\al \tau}\,d \tau,\ \ t>0.
\end{equation}
This holds for any $0<\al\le 1$, while the function $\ph_{\be}(t,\tau)$ is independent of $\al$.

As a result of the successive application of the above two steps we deduce
\begin{eqnarray}
S_{\al,\be}(t)&=&\int_0^\infty \ph_{\be}(t,\si) \int_0^\infty f_{\al}(\si,\tau) S_{1,1}(\tau)\,d \tau \,d \si\nonumber\\
&=&\int_0^\infty\left(\int_0^\infty \ph_{\be}(t,\si)f_{\al}(\si,\tau)\,d \si\right)   S_{1,1}(\tau)\,d \tau. \nonumber
\end{eqnarray}
In this way  the subordination formula 
\begin{equation}\label{sub}
S_{\al,\be}(t)=\int_0^\infty \psi_{\al,\be}(t,\tau) S_{1,1}(\tau)\,d \tau,\ \ t>0,
\end{equation}
 is derived, where the subordination kernel $\psi_{\al,\be}(t,\tau)$ admits the representation
\begin{equation}\label{psidef}
\psi_{\al,\be}(t,\tau)=\int_{0}^{\infty} \ph_\be(t,\si) f_\al(\si,\tau)\,d\si.
\end{equation}
In the scalar case the subordination identity (\ref{sub}) reduces to
\begin{equation}\label{subscalar}
E_\be(-\la^\al t^\be)=\int_0^\infty \psi_{\al,\be}(t,\tau) e^{-\la\tau}\,d \tau,\ \ t>0.
\end{equation}
 The obtained in this way Laplace transform pair for the subordination kernel $\psi_{\al,\be}(t,\tau)$ with respect to the variable $\tau$ can as well be derived  from the definition (\ref{psidef}), by the use of (\ref{subas}) and (\ref{subabs}).

To find the Laplace transform of the subordination kernel $\psi_{\al,\be}(t,\tau)$ with respect to the variable $t$, we note first that (\ref{subscalar}) and (\ref{pairs}) imply
\begin{equation}\label{psipair3}
\LL^2\{\psi_{\al,\be}(t,\tau) ; t\to s, \tau\to \la\}=\LL\{E_\be(-\la^\al t^\be) ; t\to s\}=\dfrac{s^{\be-1}}{s^\be+\la^\al}.
\end{equation}
 Then, taking inverse Laplace transform $\LL^{-1}\left\{\cdot ; \la\to \tau\right\}$ in (\ref{psipair3}) we deduce again by the use of (\ref{pairs}): 
\begin{equation}\label{psipair2}
\LL\{ \psi_{\al,\be}(t,\tau) ; t\to s\}=
s^{\be-1}\tau^{\al-1}E_{\al,\al}(-s^\be \tau^\al).
\end{equation}

In the limiting case $\al=1$ and $\be=1$ the subordination kernels are Dirac delta functions 
\begin{equation}\label{pc00}
f_1(t,\tau)=\ph_1(t,\tau)=\psi_{1,1}(t,\tau)=\de(t-\tau).
\end{equation}
Moreover, the kernels $f_\al(t,\tau)$ and $\ph_\be(t,\tau)$ are particular cases of the composite kernel $\psi_{\al,\be}(t,\tau)$, namely
\begin{equation}\label{pc}
f_\al(t,\tau)=\psi_{\al,1}(t,\tau),\ \ \ \ph_\be(t,\tau)=\psi_{1,\be}(t,\tau).\ \ 
\end{equation}
Therefore, the Laplace transform pairs for $f_\al(t,\tau)$ and $\ph_\be(t,\tau)$ can be derived from the identities (\ref{subscalar}), (\ref{psipair3}), and (\ref{psipair2}), taking $\be=1$ or $\al=1$, respectively. 




\begin{remark} It is worth noting that the integral expression in (\ref{psidef}) is not commutative: $
\psi_{\al,\be}(t,\tau) \not\equiv \int_{0}^{\infty}f_\al(t,\si) \ph_\be(\si,\tau) \,d\si$. 
This is due to the fact that the order of the two steps in the derivation procedure of subordination identity (\ref{sub}) is essential.
\end{remark}

For example, let us consider the case $\al=\be=1/2$, in which the subordination kernels can be expressed in terms of elementary functions as follows (e.g. \cite{Laplace, Yosida}): 
	\begin{equation}\label{pf}
	f_{1/2}(t,\tau)=\frac{t e^{-t^2/4\tau}}{2\sqrt{\pi}\tau^{3/2}},\ \ \ph_{1/2}(t,\tau)=\frac{1}{\sqrt{\pi t}}e^{-\tau^2/4t}.
	\end{equation}
		Plugging expressions (\ref{pf}) in the composition rule (\ref{psidef})  we get
	\begin{equation}\label{psipsi}
	\psi_{1/2,1/2}(t,\tau)=\int_{0}^{\infty} \ph_{1/2}(t,\si) f_{1/2}(\si,\tau)\,d\si=\frac{\sqrt{t}}{\pi\sqrt{\tau}(t+\tau)}.
	\end{equation}
The last formula can also be directly derived from eq. (\ref{aa}).
On the other hand, 
\begin{equation}\label{psipsi99}
\int_{0}^{\infty}f_{1/2}(t,\si) \ph_{1/2}(\si,\tau) \,d\si=\frac{2t}{\pi(t^2+\tau^2)},
\end{equation}
which can be obtained by introducing a new integration variable $(t^2+\tau^2)/4\si$. 

A comparison of identities (\ref{psipsi}) and (\ref{psipsi99}) confirms the non-commutativity pointed out in Remark~1.


Next we establish some properties of the subordination kernels based on the Laplace transform pairs (\ref{subscalar}), (\ref{psipair3}), (\ref{psipair2}), and identities (\ref{pc}).

\begin{proposition}\label{p1}
The function 
$\psi_{\al,\be}(t,\tau)$, $0<\al,\be\le 1$, is a unilateral probability density in $\tau$, i.e. it satisfies (\ref{pdf}).
\end{proposition}
\begin{proof}The normalization identity can be derived by letting $\la\to 0$ in (\ref{subscalar}):
$$
\int_0^\infty \psi_{\al,\be}(t,\tau) \,d \tau=E_\be(0)=1.
$$ The nonnegativity of the function $\psi_{\al,\be}(t,\tau)$ can be established from its Laplace transform (\ref{subscalar}) by applying Bernstein's theorem. Indeed, the Mittag-Leffler function $E_\be(-\la^\al t^\be)$ is completely monotone as a function of $\la>0$ for any fixed $t>0$ (as a composition of the completely monotone function $E_\be(-ax)$ and the Bernstein function $\la^\al$).  Definitions of these classes of functions are given in the Appendix.
\end{proof}

\begin{remark} Proposition~\ref{p1} and the subordination identity (\ref{sub}) imply that if $-A$ generates a bounded $C_0$-semigroup, i.e. $\|S_{1,1}(t)\|\le M$, $t\ge 0$, then $\|S_{\al,\be}(t)\|\le M$ for any $t\ge 0$, $0<\al\le 1$ and $0<\be\le 1$. 
In addition, if $X$ is an ordered Banach space, positivity of the $C_0$-semigroup $S_{1,1}(t)$ generated by the operator $-A$ implies positivity of the solution operator $S_{\al,\be}(t)$ for any $0<\al\le 1$ and $0<\be\le 1$. 
 \end{remark}

\begin{proposition}\label{p2}
Assume $0<\al,\al',\be,\be'\le 1$. Then
\begin{eqnarray}
f_{\al\al'}(t,\tau)&=&\int_{0}^{\infty} f_\al(t,\si) f_{\al'}(\si,\tau)\,d\si,
\label{ffab}\\
\ph_{\be\be'}(t,\tau)&=&\int_{0}^{\infty} \ph_\be(t,\si) \ph_{\be'}(\si,\tau)\,d\si.
\label{phphab}
\end{eqnarray}
\end{proposition}
\begin{proof}
To prove (\ref{ffab}) we apply Laplace transform with respect to $\tau$ 
and obtain by using (\ref{subscalar}) with $\be=1$ and Fubini's theorem $\LL\{f_{\al\al'}(t,\tau);\tau\to \la\}=e^{-t\la^{\al\al'}}$ and
$$
\LL\left\{ \int_{0}^{\infty} f_\al(t,\si) f_\al'(\si,\tau)\,d\si ; \tau\to \la\right\}=\int_{0}^{\infty} f_\al(t,\si) e^{-\si\la^{\al'}}\,d\si=e^{-t\la^{\al\al'}}.
$$
To prove (\ref{phphab}) we apply double Laplace transform, which gives by using  (\ref{subscalar}) and (\ref{psipair3}) with $\al=1$ 
\begin{eqnarray}
&&\LL^2\left\{ \int_{0}^{\infty} \ph_\be(t,\si) \ph_{\be'}(\si,\tau)\, d\si ; t\to s, \tau\to \la\right\}\nonumber\\
&&=s^{\be-1}\int_{0}^{\infty} e^{-\si s^\be} E_\be(-\la \si^{\be'})\, d\si=s^{\be-1}\frac{s^{\be({\be'}-1)}}{s^{\be\be'}+\la}=\frac{s^{\be\be'-1}}{s^{\be\be'}+\la}\nonumber\\
&&=\LL^2\{ \ph_{\be\be'}(t,\tau) ; t\to s, \tau\to \la\}.\nonumber
\end{eqnarray}
To finish the proof it remains to apply the uniqueness property of Laplace transform.
\end{proof}


Let us note that (\ref{ffab}) is equivalent to  the following natural operator identity 
\begin{equation}\label{AAA}
(A^\al)^{\al'}=(A^{\al'})^\al=A^{\al\al'},\ \ \ 0<\al,\al'\le 1,
\end{equation}
for a generator $-A$ of a bounded $C_0$-semigroup. Indeed, (\ref{ffab}) together with subordination formula (\ref{suba})  shows that any of the operators  in (\ref{AAA}) is infinitesimal generator of one and the same semigroup: $S_{\al\al',1}(t)$. 
For a different proof of (\ref{AAA}) see e.g. \cite{Yosida}, Chapter~IX. 

Identity (\ref{phphab}) is related to successive application of the subordination principle for time-fractional evolution equations and is in agreement with \cite{Baj}, Theorem~3.1. 

Next we prove that the subordination kernel $\psi_{\al,\be}(t,\tau)$, considered as a function of $t$, admits a bounded analytic extension to a sector in the complex plane. 
To this end we apply the following statement (\cite{Pruss}, Theorem 0.1.): 

If $G$ is a function defined on $(0,\infty)$ and $\theta_0\in (0,\pi/2]$ then the assertions (i) and (ii) are equivalent:

(i) $G(s)$ admits analytic extension to the sector $|\arg s|<\pi/2+\theta_0$ and $sG(s)$ is bounded on each sector $|\arg s|\le \pi/2+\theta$, $\theta<\theta_0$;

(ii) there is a function $g(t)$ analytic for $|\arg t|<\theta_0$ and bounded on each sector  $|\arg t|\le\theta<\theta_0$, such that $G(s)=\LL\{g(t); t\to s\}$ for each $s>0$.

\begin{proposition}\label{p3}
Assume $0<\al,\be\le 1$, $\al+\be<2$, and let
\begin{equation}\label{theta}
\theta_0=\min\left\{\frac{(2-\al-\be)\pi}{2\be},\frac\pi 2\right\}.
\end{equation}
For any $\tau>0$ the function $\psi_{\al,\be}(t,\tau)$ as a function of $t$ admits analytic extension to the sector $|\arg t|<\theta_0$, which is bounded on each sector $|\arg t|\le\theta$, $0<\theta<\theta_0$. 
\end{proposition}
\begin{proof}
Let us take the Laplace transform pair (\ref{psipair2}): $G(s)=s^{\be-1}\tau^{\al-1}E_{\al,\al}(-s^\be \tau^\al)$, and $g(t)=\psi_{\al,\be}(t,\tau)$, where $\tau>0$ is considered as a parameter. The function $G(s)$ admits analytic extension to $\CC$ cut along the negative real axis. According to the estimate (\ref{MLest}) for the Mittag-Leffler function
$$
|sG(s)|\le C\tau^{-1}\frac{\tau^\al|s|^\be }{1+\tau^\al|s|^\be}< C\tau^{-1},
$$
for all $s\in\CC$ such that 
$$
|\arg(s)|\le \min\left\{\frac{(2-\al)\pi}{2\be}-\e,\pi\right\}.
$$
To obtain the desired statement it remains to apply implication (i) $\Rightarrow$ (ii).
\end{proof}

\begin{definition} A solution operator $S(t)$ is said to be a bounded analytic solution operator of analyticity type $\theta_0\in (0,\pi/2]$ if $S(t)$ admits an analytic extension to the sector $|\arg t|<\theta_0$,  which is  bounded on each sector $|\arg t|\le\theta<\theta_0$, $\theta<\theta_0$.
\end{definition}

Proposition~\ref{p3} together with subordination formula (\ref{sub}) implies that $S_{\al,\be}(t)$ is a bounded analytic solution operator according to the above definition.
 The proof is similar to the one in \cite{Baj}, Theorem~3.3, where, based on analogous property for the function $\ph_\be(t,\tau)$, analyticity of the subordinated solution operator for the time-fractional evolution equation is established. 

Taking into account relations (\ref{pc}) we can derive the corresponding sectors of existence of bounded analytic extensions for the functions $f_\al(t,\tau)$ and $\ph_\be(t,\tau)$ (setting in (\ref{theta}) $\be=1$ or $\al=1$). 
In this way known results for analyticity of the semigroup $S_{\al,1}(t)$ \cite{Yosida, MiaoLi} and of the solution operator $S_{1,\be}(t)$ \cite{Baj, MiaoLi} are recovered.
The analyticity of the solution operator $S_{\al,\be}(t)$ is discussed in \cite{MiaoLi}.

\section{Relation to L\'evy extremal stable density and Mainardi function }

In this section some representation formulae for the subordination kernel $\psi_{\al,\be}(t,\tau)$ are derived in terms of  the L\'evy extremal stable density and the Mainardi function. 

We start with scaling laws for the subordination kernels. 
From the definitions (\ref{fdef}), (\ref{phidef}) and (\ref{psidef}) of the subordination kernels $f_\al,\ \ph_\be$ and $\psi_{\al,\be}$ we derive the following self-similarity properties
\begin{eqnarray}
&&f_\al(t,\tau)=t^{-1/\al}f_\al(1,\tau t^{-1/\al}),\ \ \ph_\be(t,\tau)=t^{-\be}\ph_\be(1,\tau t^{-\be}),\label{aaa}\\
&&\psi_{\al,\be}(t,\tau)=t^{-\be/\al}\psi_{\al,\be}(1,\tau t^{-\be/\al}).\label{bbb}
\end{eqnarray}
Introducing the functions of one variable $L_\al$, $M_\be$ and $K_{\al,\be}$ as follows
\begin{equation}\label{000}
L_\al(r)= f_\al(1,r),\ M_\be(r)=\ph_\be(1,r), \ K_{\al,\be}(r)=\psi_{\al,\be}(1,r),
\end{equation}
we deduce from (\ref{aaa}) and (\ref{bbb}) the following representations for the subordination kernels
\begin{eqnarray}
&&f_\al(t,\tau)=t^{-1/\al}L_\al(\tau t^{-1/\al}),\ \ \ph_\be(t,\tau)=t^{-\be}M_\be(\tau t^{-\be}),\label{fphi}\\
&&\psi_{\al,\be}(t,\tau)=t^{-\be/\al}K_{\al,\be}(\tau t^{-\be/\al}).\label{psi}
\end{eqnarray}
In addition, identities (\ref{pc00}) and (\ref{pc}) imply for the new functions 
\begin{equation}\label{pc1}
L_\al(r)=K_{\al,1}(r),\  M_\be(r)=K_{1,\be}(r),\ \ L_1(r)=M_1(r)=K_{1,1}(r)=\de(r-1).
\end{equation}

From (\ref{000}) and (\ref{subscalar}) we deduce:
\begin{equation}\label{LLK}
\LL\{K_{\al,\be}(r); r\to \la\}=E_\be(-\la^\al),
\end{equation}
i.e. $K_{\al,\be}(r)$ can be defined as the inverse Laplace transform of the Mittag-Leffler function $E_\be(-\la^\al)$.
The following Laplace transform pairs for the functions $L_\al(r)$ and $M_\be(r)$ are derived from (\ref{LLK}), taking $\be=1$ and $\al=1$, respectively:
\begin{equation}\label{LL}
\LL\{L_\al(r); r\to \la\}=\exp({-\la^\al}),\ \ \LL\{M_\be(r); r\to \la\}=E_\be(-\la).
\end{equation}
Therefore, we recognize from (\ref{LL}) the L\'evy extremal stable density $L_\al(r)$ and the Mainardi function $M_\be(r)$  (see e.g. \cite{spacetimepdf}, where these two functions appear in the context of the one-dimensional space-time fractional diffusion-wave equation). 

In the next theorem we derive representations for the subordination kernel 
in terms of $L_\al(r)$ and $M_\be(r)$.


\begin{theorem}\label{p4}
The subordination kernel $\psi_{\al,\be}(t,\tau)$ is given by (\ref{psi}), where the function $K_{\al,\be}(r)$ admits the following representations 
\begin{eqnarray}
&&K_{\al,\be}(r)=\int_0^\infty \si^{-1/\al}L_\al(r \si^{-1/\al})M_\be(\si)\, d\si,\label{1}\\
&&K_{\al,\be}(r)=\int_0^\infty \si^{\be/\al}L_\al(r \si^{\be/\al})L_\be(\si)\, d\si,\label{2}\\
&&K_{\al,\be}(r)=\al r^{\al-1}\int_0^\infty \si M_\al(\si)M_\be(\si r^\al)\, d\si.\label{3}
\end{eqnarray}
Moreover, in the particular case $\al=\be$ it holds
\begin{equation}\label{P}
K_{\al,\al}(r)=
\frac1\pi \frac{r^{\al-1}\sin\al\pi}{r^{2\al}+2r^\al\cos\al\pi+1}
\end{equation}
and if $0<\al\le\be\le 1$ then 
\begin{equation}\label{4}
K_{\al,\be}(r)=\int_0^\infty \si^{-\be/\al} L_{\al/\be}(r\si^{-\be/\al})K_{\be,\be}(\si)\, d\si.
\end{equation}
Here $L_\al(r)$ is the L\'evy extremal stable density and $M_\be(r)$ is the Mainardi function, defined by the Laplace transform pairs (\ref{LL}).
\end{theorem}
\begin{proof}
Expression (\ref{1}) follows directly from (\ref{psidef}), (\ref{fphi}) and (\ref{psi}). Representations (\ref{2}) and (\ref{3}) can be deduced from (\ref{1}) after applying the formula \cite{Pollard}
\begin{equation}\label{LM}
M_\al (r)=\frac{r^{-1-1/\al}}{\al}L_\al(r^{-1/\al}).
\end{equation}
Representation (\ref{P}) follows directly from (\ref{GM}) and (\ref{LLK}) for $\al=\be$. 
To prove (\ref{4}) we find the inverse Laplace transform of $E_\be(-\la^\al)$ by
using (\ref{GM}) and the following property (see e.g. \cite{BS}): If $I(\la)=\LL\{H(r); r\to \la\}$ and $I(\la^\al)=\LL\{H_\al(r); r\to \la\}$ then 
\begin{equation}\label{BS}
H_\al(r)=
\int_0^\infty \si^{-1/\al}L_\al(r\si^{-1/\al})H(\si)\, d\si,\ \ 0<\al\le 1.
\end{equation}
(In fact, formula (\ref{BS}) can be verified by proving that Laplace transforms of both sides are equal.)
From (\ref{BS}) and (\ref{LLK}) it follows for $0<\al\le\be\le 1$
$$
K_{\al,\be}(r)=\LL^{-1}\{E_\be(-(\la^{\al/\be})^\be; \la\to r\}=\int_0^\infty \si^{-\be/\al} L_{\al/\be}(r\si^{-\be/\al})K_{\be,\be}(\si)\, d\si
$$
and the last identity is proved.
\end{proof}

Next the regularity of the function $K_{\al,\be}(r)$ is discussed briefly. We start with the asymptotic identity implied by (\ref{MLas})
\begin{equation}\label{Eb}
E_\be(-\la^\al)\sim \frac {\la^{-\al}}{\Gamma(1-\be)},\ \ \la\to +\infty,
\end{equation}
which, by (\ref{LLK}) and applying Tauberian theorem, is equivalent to
\begin{equation}\label{Ksim}
K_{\al,\be}(r)\sim \frac{r^{\al-1}}{\Gamma(\al)\Gamma(1-\be)},\ \ r\to 0+.
\end{equation}
Therefore, if $\al<1$ and $\be<1$, then $K_{\al,\be}(r)$ has a singularity at the origin: $K_{\al,\be}(r)\to +\infty$ as $r\to 0+$.
  This is in contrast with the regular behaviour of any of the functions $L_\al(r)=K_{\al,1}(r)$ and $M_\be(r)=K_{1,\be}(r)$, which can be seen again from (\ref{Ksim}), taking $\be=1$ or $\al=1$ and noting that $\Gamma(0)=\infty$.
	
	Moreover, the Laplace transforms (\ref{LL}) of $L_\al(r)$ and $M_\be(r)$ satisfy (ii), which means that these functions admit bounded analytic extensions to appropriate sectors of the complex plane. In contrast, if $\al<1$ and $\be<1$, then (\ref{Eb}) implies that $\la E_\be(-\la^\al)\to +\infty$ as $\la\to +\infty$, therefore (ii) is not satisfied, thus, $K_{\al,\be}(r)$ does not admit a bounded analytic extension to any sector of the complex plane. 
	
	In contrast to the singular behaviour of $K_{\al,\be}(r)$ for $0<\al,\be<1$, the related kernel $\psi_{\al,\be}(t,\tau)$ exhibits a regular behaviour in $t$ (see Proposition~\ref{p3}). This can be seen in the case $\al=\be$, in which there exist simple explicit expressions for these functions: formula (\ref{P}), which together with  (\ref{psi}) yields
	\begin{equation}\label{aa}
	\psi_{\al,\al}(t,\tau)=t^{-1}K_{\al,\al}(\tau t^{-1})=\frac1\pi \frac{t^\al\tau^{\al-1}\sin\al\pi}{t^{2\al}+2t^\al \tau^\al\cos\al\pi+\tau^{2\al}}.
	\end{equation}

	The next theorem summarizes the main results obtained so far. 

\begin{theorem}\label{thsub}
Let $A$ be a generator of a bounded $C_0$-semigroup $S_{1,1}(t)$. 
Then
 problem (\ref{ab}) admits a bounded analytic solution operator $S_{\al,\be}(t)$ of analyticity type at least $\theta_0$, where $\theta_0$ is defined in (\ref{theta}). 
The solution operator $S_{\al,\be}(t)$ is related to $S_{1,1}(t)$ via the subordination identity 
\begin{equation}\label{sub11}
S_{\al,\be}(t)=\int_0^\infty \psi_{\al,\be}(t,\tau) S_{1,1}(\tau)\,d \tau,\ \ t>0,
\end{equation}
 where $\psi_{\al,\be}(t,\tau)=t^{-\be/\al}K_{\al,\be}(\tau t^{-\be/\al})$, with $K_{\al,\be}(r)$ being a unilateral probability density function 
  defined as the inverse Laplace transform of a Mittag-Leffler function: 
\begin{equation}\label{LK}
E_\be(-\la^\al)=\int_0^\infty e^{-\la r} K_{\al,\be}(r)\, dr.
\end{equation}
Moreover, $K_{\al,\be}(r)$ admits the representations (\ref{1})-(\ref{4}), where $M_\al(r)$ is the Mainardi function  and $L_\al(r)$ is the L\'evy one-sided extremal stable density, defined by the Laplace transform pairs (\ref{LL}).
\end{theorem}

\begin{corollary}\label{cor}
Let $0<\al\le\be<1$.
Under the conditions of Theorem~\ref{thsub} the solution operator $S_{\al,\be}(t)$ admits the representation
\begin{equation}\label{subbebe}
S_{\al,\be}(t)=\int_0^\infty \psi_{\be,\be}(t,\tau) S_{\al/\be,1}(\tau)\,d \tau,\ \ t>0,
\end{equation}
where $S_{\al/\be,1}(t)$ is the $C_0$-semigroup generated by the operator $-A^{\al/\be}$ and the function $\psi_{\be,\be}(t,\tau)$ is defined in (\ref{aa}). 
\end{corollary}
\begin{proof}
This subordination identity is derived by plugging in (\ref{sub11}) the identity
$$
f_{\al}(t,\tau)=\int_{0}^{\infty} f_\be(t,\si) f_{\al/\be}(\si,\tau)\,d\si
$$
following from (\ref{ffab}), and applying (\ref{psidef}), Fubini's theorem, and (\ref{suba}).
\end{proof}

 The considerations in this section reveal a relation between the subordination kernel $\psi_{\al,\be}(t,\tau)$ and the fundamental solution of the one-dimensional space-time fractional diffusion-wave equation
\begin{equation}\label{sk}
\DD_t^\be u(x,t)=D_{x}^{\al,\theta}u(x,t),\ \ x\in\RR, t>0,
\end{equation}
where $D_{x}^{\al,\theta}$ is the Riesz-Feller space-fractional derivative of order $\al$ and skewness $\theta$, see e.g. \cite{spacetimepdf, MainardiPagniniGorenflo, Pagnini}.
Indeed, a comparison to the representations of the fundamental solution to problem (\ref{sk}) in terms of Mainardi function and L\'evy extremal stable density, obtained in \cite{spacetimepdf}, or subordination formula (56) in \cite{Pagnini}, imply the following relation
\begin{equation*}
\psi_{\al,\be}(t,\tau)=2\left.\GG_{\al,\be}^{-\al}(x,t)\right|_{x=\tau},
\end{equation*}
where $\GG_{\al,\be}^{-\al}(x,t)$ is the Green function (restricted to the $x$-range $x>0$) of the space-time fractional equation (\ref{sk}) with Caputo time-derivative of order $\be$ and Riesz-Feller space-derivative of order $\al$  and skewness $-\al$.

\section{Applications}

\subsection{Space-time fractional diffusion equation on $\RR^n$}

 The first example concerns  the space-time fractional diffusion equation (\ref{st}) on $\RR^n$, studied in a series of papers, e.g. \cite{spacetimepdf, Luchko, LuchkoMMNP, LuchkoAxioms, Luchkosub}. Subordination formulae for this equation have already been derived by the use of Mellin transform, see e.g. \cite{MainardiPagniniGorenflo, Pagnini} for $n=1$ and \cite{Luchkosub} for the multi-dimensional equation.   

Assume $X$ is one of the spaces $L^p(\RR^n)$, $1< p<\infty$, or the space $C_0(\RR^n)$ of continuous functions vanishing at infinity  and let $A$ be the negative Laplace operator defined on $X$ with maximal domain. For details on the definition of the full-space Laplace operator we refer to \cite{Martinez}. 
The operator $A$ is a generator of a bounded analytic $C_0$-semigroup $S_{1,1}(t)$ with corresponding Green function
  \cite{Laplace}
	\begin{equation}\label{Gaus}
	G_{1,1,n}(x,t)=\frac{1}{(4\pi t)^{n/2}}e^{-|x|^2/4t},\ \ x\in\RR^n,\ t>0.
	\end{equation}

For any $0<\al\le 1$ and $0<\be\le 1$ the solution operator $S_{\al,\be}(t)$ of problem (\ref{ab}) is given by
	$$
	(S_{\al,\be}(t)v)(x)=\int_{\RR^n} G_{\al,\be,n}(y,t) v(x-y)\, dy,\ \ v\in X,\ \ t>0,x\in\RR^n,
	$$
	where $G_{\al,\be,n}(x,t)$ is the corresponding Green function. 
	Therefore, the subordination formula (\ref{sub}) can be written in terms of Green functions as follows
	\begin{equation}\label{Gab}
	G_{\al,\be,n}(x,t)=\int_0^\infty \psi_{\al,\be}(t,\tau)G_{1,1,n}(x,\tau)\, d\tau.
	\end{equation}
	Based on (\ref{Gab}) and (\ref{Gaus}), we find closed form expressions for $G_{\al,\be,n}(x,t)$ for $\al=\be$. In this case Eq. (\ref{st}) is the so-called $\al$-fractional diffusion equation studied in \cite{LuchkoAxioms, LuchkoCAMWA}. 
	Taking into account (\ref{aa}), the subordination formula (\ref{Gab}) reads
	\begin{equation}\label{Gaa}
	G_{\al,\al,n}(x,t)=\frac{1}{\pi}\int_0^\infty  \frac{t^\al\tau^{\al-1}\sin\al\pi}{t^{2\al}+2t^\al \tau^\al\cos\al\pi+\tau^{2\al}}G_{1,1,n}(x,\tau)\, d\tau.
	\end{equation}
	
	For $x=0$ the integral in (\ref{Gaa}) is convergent  only if $\al>n/2$. Therefore, $G_{\al,\al,n}(0,t)$ is finite only for $n=1$ and $\al>1/2$. The same conclusion can be found in  \cite{LuchkoCAMWA, LuchkoAxioms}. 
		
		Next, applying the subordination formula (\ref{Gaa}), we derive a closed-form expression for the two-dimensional Green function.  Plugging (\ref{Gaus}) with $n=2$ in (\ref{Gaa}) yields
\begin{eqnarray}
G_{\al,\al,2}(x,t)
&=&\frac{t^\al}{4\pi^2} \int_0^\infty\frac{\tau^{\al-2}\sin\al\pi}{t^{2\al}+2t^\al \tau^\al\cos\al\pi+\tau^{2\al}}e^{-{|x|^2}/{4\tau}}\, d\tau\nonumber\\
&=&\frac{1}{4\pi t} \int_0^\infty\frac{\si^{\al}\sin\al\pi}{\si^{2\al}+2\si^\al \cos\al\pi+1}e^{-(|x|^2/4t)\si}\, d\si,\nonumber
\end{eqnarray}
where we have made the change of variables $\si=t/\tau$. Formula (\ref{GM}) gives the following expression  in terms of Mittag-Leffler functions
\begin{equation}\label{2D}
G_{\al,\al,2}(x,t)=\frac{1}{4\pi t}(|x|^2/4t)^{\al-1}E_{\al,\al}(- (|x|^2/4t)^\al).
\end{equation}
The same result can be found in \cite{LuchkoMMNP, LuchkoCAMWA, Luchkosub}. 

	Further, let us restrict our attention to the special case $\al=\be=1/2$. We will derive some results, which seem to be new. 
			 Plugging (\ref{psipsi}) in the subordination formula (\ref{Gab}) yields
	\begin{equation}\label{Gaaa}
	G_{1/2,1/2,n}(x,t)=\frac{\sqrt{t}}{\pi}\int_0^\infty  \frac{1}{\sqrt{\tau}(t+\tau)}G_{1,1,n}(x,\tau)\, d\tau.
	\end{equation}
	Inserting (\ref{Gaus}) in (\ref{Gaaa}) and introducing a new integration variable $\si=t/\tau$ gives
\begin{equation}\label{GU}
G_{1/2,1/2,n}(x,t)=\frac{\Gamma\left(\frac{n+1}{2}\right)}{2^n \pi^{n/2+1}t^{n/2}}U\left(\frac{n+1}{2},\frac{n+1}{2},\frac{|x|^2}{4t}\right),
\end{equation}
where $U$ is the Tricomi's confluent hypergeometric function (\ref{tricomi}). 

We first check that  formula (\ref{GU}) with $n=2$ is the same as (\ref{2D}) with $\al=1/2$ due to relation (\ref{MLT}).

Applying (\ref{relation}), expression (\ref{GU}) for the Green function can be rewritten in terms of the incomplete Gamma function (\ref{Gamma}) as follows
\begin{equation}\label{GG}
G_{1/2,1/2,n}(x,t)=\frac{\Gamma\left(\frac{n+1}{2}\right)}{2^n \pi^{n/2+1}t^{n/2}}e^{{|x|^2}/{4t}}\Gamma\left(\frac{1-n}{2},\frac{|x|^2}{4t}\right).
\end{equation}
In the one-dimensional case representations (\ref{GU}) and (\ref{GG}) reduce to
\begin{equation}\label{GG1}
G_{1/2,1/2,1}(x,t)=\frac{1}{2 \pi^{3/2}\sqrt{t}}e^{{x^2}/{4t}}\EEE_1\left(\frac{x^2}{4t}\right),
\end{equation}
where $\EEE_1(z)$ is the exponential integral (\ref{ei}). The asymptotic expansion (\ref{Elargez}) implies for ${x^2}/{4t}\to\infty$ (i.e. as $x\to \infty$ and $t>0$ is fixed or $t\to 0$ and $x\neq 0$ is fixed)
$$
G_{1/2,1/2,1}(x,t)\approx \frac{2\sqrt{t}}{\pi^{3/2}x^2},\ \ {x^2}/{4t}\to\infty.
$$
Similar asymptotic is observed for $\al>1/2$, see \cite{LuchkoAxioms}, Eq. 25. 

On the other hand, the expansion (\ref{Esmallz}) of the exponential integral gives for $x\in\RR, x\neq 0$ and $t>0$
$$
G_{1/2,1/2,1}(x,t)= \frac{e^{{x^2}/{4t}}}{2 \pi^{3/2}\sqrt{t}}\left(-\g-\ln \left(\frac{x^2}{4t}\right)-\sum_{k=1}^\infty \frac{\left({x^2}/{4t}\right)^k}{k (k!)}\right).
$$
This expansion implies the following asymptotic behaviour
$$
G_{1/2,1/2,1}(x,t)\approx \frac{\ln 4t-\ln x^2}{2 \pi^{3/2}\sqrt{t}},\ \ {x^2}/{4t}\to 0.
$$
Therefore, $G_{1/2,1/2,1}(x,t)\to\infty$ with logarithmic growth as $x\to 0$ for any fixed $t>0$ and $G_{1/2,1/2,1}(x,t)\to 0$ as $t\to \infty$ and $x\neq 0$ is fixed.
Let us note that when $\al>1/2$ the asymptotic behaviour of $G_{\al,\al,1}(x,t)$ as $x\to 0$ is of a power law type (see \cite{LuchkoAxioms}, Eq. 24), which is in contrast to the observed here logarithmic growth for $\al=1/2$. 

In addition,  inequalities (\ref{br}) imply that the Green function $G_{1/2,1/2,1}(x,t)$ is bracketed for any $x\in\RR, x\neq 0$, and $t>0$ as follows 
$$
\frac{1}{4 \pi^{3/2}\sqrt{t}} \ln \left( 1+\frac{8t}{x^2}\right) < G_{1/2,1/2,1}(x,t)< \frac{1}{2 \pi^{3/2}\sqrt{t}}\ln \left( 1+\frac{4t}{x^2}\right).
$$

Expressions for the multi-dimensional Green function $G_{1/2,1/2,n}(x,t)$ in terms of the exponential integral $\EEE_1$ (for odd dimensions) or the Mittag-Leffler function $E_{1/2,1/2}$ (for even dimensions) can be obtained from (\ref{GG1}) and (\ref{2D}) by applying representation (\ref{GG}) and the recurrence relation (\ref{rec}) between the incomplete Gamma functions.   For example, in this way we derive from (\ref{GG}), (\ref{GG1}), and (\ref{rec}),  the following expression for the three-dimensional Green function
$$
G_{1/2,1/2,3}(x,t)=
\frac{1}{2 \pi^{5/2}\sqrt{t}|x|^2}-\frac{1}{8\pi^{5/2} t^{3/2}}e^{{|x|^2}/{4t}}\EEE_1\left(\frac{|x|^2}{4t}\right).
$$

Finally, we present an application of the subordination formula (\ref{subbebe}), which for $\be=2\al$ gives in terms of Green functions 
\begin{equation}\label{sss}
G_{\al/2,\al,n}(x,t)=\frac{1}{\pi}\int_0^\infty  \frac{t^\al\tau^{\al-1}\sin\al\pi}{t^{2\al}+2t^\al \tau^\al\cos\al\pi+\tau^{2\al}}G_{1/2,1,n}(x,\tau)\,d \tau, 
\end{equation}
where $G_{1/2,1,n}(x,t)$ is the multi-dimensional Poisson kernel \cite{Laplace}
\begin{equation}\label{Poisson}
	G_{1/2,1,n}(x,t)=\frac{\Gamma\left(\frac{n+1}{2}\right)t}{\pi^{(n+1)/2}(t^2+|x|^2)^{(n+1)/2}}.
	\end{equation}
	For example, applying (\ref{sss}) for $n=1$, we can recover the following known closed-form expression (see e.g. \cite{spacetimepdf}, Eq. (4.38), \cite{Pagnini}, Eq. (20))
	\begin{equation}\label{neutral}
	G_{\al/2,\al,1}(x,t)=\frac{1}{\pi}\frac{t^\al x^{\al-1}\sin(\al\pi/2)}{t^{2\al}+2t^\al x^\al\cos(\al\pi/2)+\tau^{2\al}}, \ \ x>0, t>0.
	\end{equation}
	Indeed, starting from the following integral obtained from (\ref{sss}) and (\ref{Poisson})
	\begin{equation}\label{int}
	G_{\al/2,\al,1}(x,t)=\frac{1}{\pi^2}\int_0^\infty  \frac{t^\al\tau^{\al-1}\sin\al\pi}{t^{2\al}+2t^\al \tau^\al\cos\al\pi+\tau^{2\al}}\frac{\tau}{\tau^2+x^2}\,d \tau, 
	\end{equation}
	inserting in (\ref{int}) the identity
	$$
	\frac{1}{\tau^2+x^2}=\frac{1}{x}\int_0^\infty e^{-\tau\si} \sin x\si\, d\si
	$$
	and changing the order of integration we obtain by the use of (\ref{GM})
	\begin{equation}\label{intint}
	G_{\al/2,\al,1}(x,t)=\frac{t^\al}{\pi x}\int_0^\infty  \si^{\al-1}E_{\al,\al}(-t^\al \si^\al)\sin x\si\, d\si.
	\end{equation}
	Note that the integral in (\ref{intint}) is convergent due to the asymptotic expansion (\ref{MLas}) of the Mittag-Leffler function.
	Calculation of (\ref{intint}) by using the Laplace transform pair (\ref{pairs}) yields (\ref{neutral}).

\subsection{Space-time fractional diffusion equation on a bounded domain}

 As a second example we consider the space-time fractional diffusion equation on a bounded domain with Dirichlet boundary conditions. 

Assume $\Om\subset\RR^n$ is an open set and let $X=L^2(\Om)$. Let $A=-\De$, where $\De$ is the Dirichlet Laplacian, i.e. $D(A)= H_0^1(\Omega) \cap H^2(\Omega)$. It is known that the operator $-A$ generates a bounded analytic $C_0$-semigroup $S_{1,1}(t)$ (see e.g \cite{Laplace}),
which admits the following eigenfunction decomposition
\begin{equation}\label{opTTT}
S_{1,1}(t)v=\sum_{j=1}^\infty e^{-{\la_j} t} (v,\phi_j) \phi_j.
\end{equation} 
Here $(.,.)$ denotes the $L^2(\Om)$ inner product and $\phi_j$ is the $L^2(\Om)$-orthonormal basis of eigenfunctions of $-\De$ with eigenvalues $\la_j$.

Applying 
the subordination identity (\ref{sub}) we obtain by using the scalar formula  (\ref{subscalar})
the eigenfunction expansion of the solution of problem (\ref{ab0}) 
\begin{equation}\label{FS}
S_{\al,\be}(t)v=\sum_{n=1}^\infty E_\be(-\la^\al_j t^\be)(v,\phi_j) \phi_j.
\end{equation}
The same result can be obtained by direct application of eigenfunction expansion to solve the corresponding PDE, by taking into account the spectral definition (\ref{fl}) of the fractional Laplacian.

\subsection{Space-time fractional advection equation}

As a third example we consider a fractional-order generalization of the classical one-dimensional advection equation 
\begin{eqnarray}
&&\frac{\partial u}{\partial t}+\frac{\partial u}{\partial x}=0,\ \ x>0,\ t>0,\label{adveq}\\
&&u(x,0)=v(x),\ \ u(0,t)=0.\nonumber
\end{eqnarray}
Let $X=L^p(I)$,  $1\le p<\infty$, or $X=C(I)$, 
where $I=(0,T)$. Assume $A=d/dx$ with domain $D(A)$ consisting of all functions $v\in X$, such that $v'(x)\in X$ and $v(0)=0$.
The operator $A$ generates a bounded $C_0$-semigroup $S_{1,1}(t)$, defined by the solution of problem (\ref{adveq})  
\begin{equation}\label{sol0}
(S_{1,1}(t)v)(x)=H(x-t)v(x-t),
\end{equation}
where $H(\cdot)$ denotes the Heaviside unit-step function.

The classical advection equation (\ref{adveq}) is generalized as follows. As in the previous examples we take a fractional Caputo time-derivative $\DD_t^\be$, $0<\be\le 1$. In space we replace the first derivative by  a fractional Riemann-Liouville derivative $D_x^\al$, $0<\al\le 1$. Motivation for this is the fact that the Riemann-Liouville differential operator 
with domain
$
\{v\in X|\ \ D_x^\al v\in X,\ \ (J_x^{1-\al}v)(0)=0\}
$
can be considered as a fractional power of the operator $A$ defined above (see \cite{Baj}, Lemma~1.8., and \cite{Ashyralyev}). 
For the definitions of the fractional-order integro-differential operators $\DD_t^\be, D_x^\al, J_x^{\al}$ see Appendix. 

Therefore, we consider the following space-time fractional advection equation 
\begin{eqnarray}
&&\DD_t^\be u+D_x^\al u=0,\ \ x>0,\ t>0,\ 0<\al\le 1,\ 0<\be\le 1,\label{fradveq}\\
&&u(x,0)=v(x),\ \ J_x^{1-\al}u|_{x=0}=0.\nonumber
\end{eqnarray}
Applying subordination identity (\ref{sub}) we derive from (\ref{sol0}) the solution of the space-time fractional advection equation (\ref{fradveq})
\begin{equation}\label{sol1}
u(x,t)=(S_{\al,\be}(t)v)(x)=\int_0^x\psi_{\al,\be}(t,\tau)v(x-\tau)\, d\tau.
\end{equation}
The same result can be obtained by applying Laplace transform both in space and in time to solve problem (\ref{fradveq}).

The solution operator $S_{\al,\be}(t)$ in (\ref{sol1}) admits a bounded analytic extension to the sector $|\arg t|<\theta_0$, where $\theta_0$ is defined in (\ref{theta}). This is implied by the same property of the subordination kernel  $\psi_{\al,\be}(t,\tau)$ appearing in (\ref{sol1}).
The solution operators in the previous examples admit bounded analytic extensions to larger sectors than the mentioned in Theorem~\ref{thsub} (e.g. the Green function (\ref{GG1}) is bounded analytic for $|\arg t|<\pi$). This is due to the fact that the original semigroup $S_{1,1}(t)$ in the two preceding examples is already analytic. 

\section{Concluding remarks}

 In this work the subordination principle for space-time fractional evolution equations is analyzed. The subordination kernel is expressed in terms of Mainardi function and L\'evy extremal stable density, and coincides with the fundamental solution (restricted to $x>0$) of the one-dimensional Cauchy problem for the space-time fractional diffusion equation with the space-fractional Riesz-Feller derivative of order $\al$ and skewness $-\al$.
The existence of bounded analytic extension for the subordination kernel is studied. Applying the obtained subordination formulae some known results are recovered. Application to the multi-dimensional $\al$-fractional diffusion equation with $\al=1/2$ leads to a closed-form expression of the fundamental solution in terms of Tricomi's confluent hypergeometric function. 
In the one-dimensional case the expression reveals logarithmic infinite growth of the solution when $x\to 0$ for any fixed $t>0$. 

The potential of the subordination principle in the study of multi-dimensional space-time fractional diffusion equations will be further exploited in future works.

\section*{Acknowledgment}
 This work is supported by project ``Analytical and numerical methods for differential and integral equations..."
under bilateral agreement between Serbian Academy of Sciences and Arts and Bulgarian Academy of Sciences.

\section*{Appendix}  
 
The Caputo and the Riemann-Liouville fractional derivatives of order $\al\in(0,1)$, denoted in this work by $\DD_t^\al$ and $D_t^\al$, respectively,
are defined as follows 
\begin{equation*}
 \DD_t^\al =  J_t^{1-\al} (d/dt),\ \   D_t^\al = (d/dt) J_t^{1-\al},\ \ \ 0<\al<1,
\end{equation*} 
where $J_t^\al$ is the fractional Riemann-Liouville integral $$ J_t^\al f(t)=\frac{1}{\Gamma(\al)}\int_0^t (t-\tau)^{\al-1}f(\tau)\,d\tau,\ \ \al>0,$$
with $\Ga(\cdot)$ being the Gamma function. 

The two-parameter Mittag-Leffler function is an entire function defined by the series \cite{bookP, KST, GM}
\begin{equation*}
E_{\al,\be}(z)=\sum_{k=0}^{\infty}\frac{z^k}{\Ga(\al k+\be)},\ \ E_\al(z)=E_{\al,1}(z),
\ \ \al,\be, z\in \CC,\ \Re \al>0.
\end{equation*}
For $0<\al<2, \be\in\RR$, and $\mu$ such that $\al\pi/2<\mu<\min\{\pi,\al\pi\}$, the following asymptotic expansion holds \cite{bookP,KST}
\begin{equation}\label{MLas}
E_{\al,\be}(z)=-\sum_{k=1}^{N-1}\frac{z^{-k}}{\Gamma(\be-\al k)}+O(|z|^{-N}),\ \ \mu\le |\arg z|\le\pi,\ \ |z|\to \infty.
\end{equation}
This asymptotic expansion implies the estimate (\cite{bookP}, Theorem~1.6)
\begin{equation}\label{MLest}
\left|E_{\al,\be}(z)\right|\le\frac{C}{1+|z|},\ \ \mu\le |\arg z|\le\pi, |z|\ge 0.
\end{equation}
Recall the Laplace transform pair \cite{KST} 
\begin{equation}\label{pairs}
\mathcal{L}\left\{t^{\be-1}E_{\al,\be}(-\la t^\al);t \to s\right\}=\frac{s^{\al-\be}}{s^\al+\la},\ \ \ \Re\al>0, \ \la\in\RR, \ t>0. 
\end{equation}
We point out the following representation of the functions of Mittag-Leffler type $t^{\be-1}E_{\al,\be}(-\la t^\al)$ as Laplace transform of a spectral function $P_{\al,\be}$ (see
\cite{GM})
\begin{equation}\label{GM}
t^{\be-1}E_{\al,\be}(-\la t^\al)=\int_0^\infty e^{-rt} P_{\al,\be}(r;\la) \, dr,\ \ \ 
\end{equation}
where $\la>0$, $0<\al\le\be\le 1$, excluding the case $\al=\be=1$, and
$$
P_{\al,\be}(r;\la)= \frac{1}{\pi}\frac{ r^{\al}\sin\be\pi+\la\sin(\be-\al)\pi}{r^{2\al}+2\la r^\al\cos\al\pi+\la^2}r^{\al-\be}\ge 0.
$$

A function $f\in C^\infty(0,\infty)$ is said to be completely monotone if 
$$
(-1)^n f^{(n)}(t)\ge 0, \mbox{\ for\ all\ } t> 0, \ n=0,1,...
$$
The functions $e^{-a t^\al}$, $a\ge 0$, $0<\al\le 1$ and the Mittag-Leffler function $E_{\al,\be}(-t)$, $0<\al\le 1$, $\be\ge\al$, are well-known examples of completely monotone functions for $t>0$.
A $C^\infty$ function $g:(0,\infty)\to\RR$ is said to be a Bernstein function if it is nonnegative and its first derivative $g'(t)$ is a completely monotone function. 
The composition $f(g)$ of a completely monotone function $f$ with a Bernstein function $g$ is again completely monotone.

 The characterization of completely monotone functions is given by the Bernstein's theorem which states that 
a function $f:(0,\infty)\to\RR$ is completely monotone if and only if it can be represented as the Laplace transform of a nonnegative measure.
Bernstein's theorem and representation (\ref{GM}) show the complete monotonicity of the function of Mittag-Leffler type 
$t^{\be-1}E_{\al,\be}(-\la t^\al)$ for $0<\al\le\be\le 1$, since then $P_{\al,\be}(r;\la)\ge 0$, for any $r,\la>0$.

For more details on Mittag-Leffler-type functions we refer to \cite{KST, bookML, Bessel, PanevaFCAA18}. For the theory of completely monotone and Bernstein functions see \cite{CMF}.

The Tricomi's confluent hypergeometric function can be defined by the Laplace integral (\cite{AS}, Eq. 13.2.5)
\begin{equation}\label{tricomi}
U(a,c,z)=\frac{1}{\Gamma(a)}\int_0^\infty \xi^{a-1}(1+\xi)^{c-a-1} e^{-z\xi}\, d\xi, \ a>0, z>0.
\end{equation}
The following representations of Mittag-Leffler functions in terms of Tricomi's confluent hypergeometric function can be obtained from (\ref{GM}) for $\al=\be=1/2$   
\begin{equation}\label{MLT}
E_{1/2}(-t^{1/2})=\frac{1}{\sqrt\pi} U(1/2,1/2,t);\ \ t^{-1/2}E_{1/2,1/2}(-t^{1/2})=\frac{1}{2\sqrt\pi} U(3/2,3/2,t).
\end{equation}
For $a=c$ the Tricomi's confluent hypergeometric function (\ref{tricomi}) is related to the upper incomplete Gamma function
\begin{equation}\label{Gamma}
\Gamma(a,z)=\int_z^\infty \xi^{a-1} e^{-\xi}\, d\xi
\end{equation}
as follows (\cite{AS}, Eq. 13.6.28)
\begin{equation}\label{relation}
U(a,a,z)=e^z\Gamma(1-a,z).
\end{equation}
Integration by parts in (\ref{Gamma}) yields the following recurrence relation
\begin{equation}\label{rec}
\Gamma(a+1,z)=z^a e^{-z}+a\Gamma(a,z).
\end{equation}
The incomplete Gamma function (\ref{Gamma}) with $a=0$ gives the exponential integral  
\begin{equation}\label{ei}
\EEE_1(z)=\int_z^\infty\frac{e^{-\xi}}{\xi}\, d\xi,
\end{equation}
which satisfies $\EEE_1(z)=\Gamma(0,z)=e^{-z}U(1,1,z)$. For real or complex arguments off the negative real axis, it can be expressed as (\cite{AS}, Eq. 5.1.11)
	\begin{equation}\label{Esmallz}
	\EEE_1(z)=-\g-\ln z-\sum_{k=1}^\infty \frac{(-z)^k}{k (k!)},\ \ |\arg z|<\pi,
	\end{equation}
where $\g$ is the Euler constant. For large values of $\Re z$ the following approximation is valid \cite{BH}
\begin{equation}\label{Elargez}
	\EEE_1(z)\approx\frac{e^{-z}}{z}\sum_{k=0}^{N-1} \frac{ k!}{(-z)^k}. 
	\end{equation}
	For real positive values of the argument the exponential integral can be bracketed by elementary functions as follows (\cite{AS}, Eq. 5.1.20)
	\begin{equation}\label{br}
	 0.5 e^{-x }\ln \left( 1+2/x\right) < \EEE_1(x)< e^{-x }\ln \left( 1+1/x\right),\ \ x>0.
	\end{equation}



\end{document}